\documentclass[10pt,twoside,a4paper]{article}
\usepackage{amsfonts,amssymb,amsmath,amscd,latexsym,makeidx,theorem}
\author{Luca Martinazzi\thanks{This work was supported by ETH Research Grant no. ETH-02 08-2.}
\\ \small ETH Zurich \\ \small R\"amistrasse 101, CH-8092 \\ \small \texttt{luca@math.ethz.ch}
}
\title{A threshold phenomenon for embeddings of $H^m_0$ into Orlicz spaces}

\newtheorem{trm}{Theorem}
\newtheorem{prop}[trm]{Proposition}
\newtheorem{cor}[trm]{Corollary}
\newtheorem{lemma}[trm]{Lemma}

\theorembodyfont{\upshape}

\theorembodyfont{\upshape}

\theorembodyfont{\upshape}

\newcommand{\R}[1]{\mathbb{R}^{#1}}

\newcommand{\de}{\partial}

\newenvironment{proof}{\noindent\emph{Proof.}}{\hfill$\square$\medskip}

\newenvironment{rmk}{\medskip\noindent\emph{Remark.}}{\medskip}

\DeclareMathOperator{\diver}{div}
\DeclareMathOperator{\loc}{loc}

\DeclareMathOperator*{\dist}{dist}

\DeclareMathOperator*{\Intm}{\int\!\!\!\!\!\! \rule[2.6pt]{6.5pt}{.4pt}}

\begin{document}
\maketitle

\begin{abstract}
Given an open bounded domain $\Omega\subset\R{2m}$ with smooth boundary, we consider a sequence $(u_k)_{k\in\mathbb{N}}$ of positive smooth solutions to 
$$
\left\{
\begin{array}{ll}
(-\Delta)^m u_k=\lambda_k u_k e^{mu_k^2} & \text{in }\Omega\\
u_k=\de_\nu u_k=\ldots =\de_\nu^{m-1} u_k=0 &\text{on } \de \Omega,
\end{array}
\right.
$$
where $\lambda_k\to 0^+$. Assuming that the sequence is bounded in $H^m_0(\Omega)$, we study its concentration-compactness behavior. We show that if the sequence is not precompact, then
$$\liminf_{k\to\infty}\|u_k\|^2_{H^m_0}:=\liminf_{k\to\infty}\int_\Omega u_k(-\Delta)^m u_k dx\geq \Lambda_1,$$
where $\Lambda_1=(2m-1)!\mathrm{vol}(S^{2m})$ is the total $Q$-curvature of $S^{2m}$.
\end{abstract}

\section{Introduction and statement of the main result}

\medskip

Let $\Omega\subset\R{2m}$ be open, bounded and with smooth boundary, and let a sequence $\lambda_k\to 0^+$ be given. Consider a sequence $(u_k)_{k\in\mathbb{N}}$ of smooth solutions to
\begin{equation}\label{eq0}
\left\{
\begin{array}{ll}
(-\Delta)^m u_k =\lambda_k u_k e^{m u_k^2} &\textrm{in } \Omega\\
u_k>0 & \textrm{in }\Omega \rule{0cm}{0.4cm}\\
u_k=\partial_\nu u_k=\ldots =\partial_\nu^{m-1} u_k=0& \textrm{on } \partial\Omega. \rule{0cm}{0.4cm} 
\end{array}
\right.
\end{equation}
Assume also that
\begin{equation}\label{eq1}
\int_\Omega u_k(-\Delta)^m u_k dx=\lambda_k\int_\Omega u_k^2 e^{m u_k^2}dx\to\Lambda \geq 0\quad \textrm{as } k\to \infty.
\end{equation}
In this paper we shall prove

\begin{trm}\label{trm1} Let $(u_k)$ be a sequence of solutions to \eqref{eq0}, \eqref{eq1}. Then either

\medskip

\noindent (i) $\Lambda=0$ and $u_k\to 0$ in $C^{2m-1,\alpha}(\Omega)$,\footnote{Here and in the following $\alpha\in [0,1)$ is an arbitrary H\"older exponent.} or
 
\medskip
 
\noindent (ii) We have $\sup_\Omega u_k\to \infty$ as $k\to\infty$. Moreover there exists $I\in\mathbb{N}\backslash\{0\}$ such that $\Lambda\geq I\Lambda_1$, where $\Lambda_1:=(2m-1)!\mathrm{vol}(S^{2m})$,  and up to a subsequence there are $I$ converging sequences of points $x_{i,k}\to x^{(i)}$ and of positive numbers $r_{i,k}\to 0$, the latter defined by
\begin{equation}\label{rik}
\lambda_k r_{i,k}^{2m} u_k^2(x_{i,k}) e^{m u_k^2(x_{i,k})}=2^{2m}(2m-1)!,
\end{equation}
such that the following is true:
\begin{enumerate}
\item If we define
$$\eta_{i,k}(x):=u_k(x_{i,k})(u_k(x_{i,k}+r_{i,k}x)-u_k(x_{i,k}))+\log 2$$
for $1\leq i\leq I$, then 
\begin{equation}\label{etak}
\eta_{i,k}(x)\to \eta_0(x)=\log\frac{2}{1+|x|^2}\quad \textrm{in } C^{2m-1}_{\loc}(\R{2m})\quad (k\to\infty).
\end{equation}
\item For every $1\leq i\neq j \leq I$ we have $\frac{|x_{i,k}-x_{j,k}|}{r_{i,k}}\to\infty$ as $k\to\infty$.

\item Set $R_k(x):=\inf_{1\leq i\leq I}|x-x_{i,k}|$. Then
\begin{equation}\label{Rk}
\lambda_k R_k^{2m}(x) u_k^2(x)e^{m u_k^2(x)}\leq C,
\end{equation}
where $C$ does not depend on $x$ or $k$.
\end{enumerate}
Finally $u_k\rightharpoonup 0$ in $H^m(\Omega)$ and $u_k\to 0$ in $C^{2m-1,\alpha}_{\loc}(\overline\Omega\backslash\{x^{(1)},\ldots, x^{(I)}\})$.
\end{trm}

Solutions to \eqref{eq0} arise from the Adams-Moser-Trudinger inequality \cite{ada}:
\begin{equation}\label{ada}
\sup_{u\in H^m_0(\Omega),\, \|u\|_{H^m_0}^2\leq \Lambda_1}\Intm_\Omega e^{mu^2}dx= c_0(m)<+\infty,
\end{equation}
where $c_0(m)$ is a dimensional constant, and $H^m_0(\Omega)$ is the Beppo-Levi defined as the completion of $C^\infty_c(\Omega)$ with respect to the norm\footnote{The norm in \eqref{norm} is equivalent to the usual Sobolev norm
$\|u\|_{H^m}:=\left(\sum_{\ell=0}^m\|\nabla^\ell u\|_{L^2}\right)^\frac{1}{2},$
thanks to elliptic estimates.}
\begin{equation}\label{norm}
\|u\|_{H^m_0}:=\|\Delta^\frac{m}{2}u\|_{L^2}=\bigg(\int_\Omega |\Delta^\frac{m}{2}u|^2dx\bigg)^\frac{1}{2},
\end{equation}
and we used the following notation:
\begin{equation}\label{Delta}
\Delta^\frac{m}{2}u:=\left\{
\begin{array}{ll}
\Delta^n u\in\R{} &\text{if } m=2n\text{ is  even},\\
\nabla \Delta^n u\in\R{2m} &\text{if } m=2n+1\text{ is  odd}.\\
\end{array}
\right.
\end{equation}
In fact \eqref{eq0} is the Euler-Lagrange equation of the functional
$$F(u):=\frac{1}{2}\int_\Omega |\Delta^\frac{m}{2}u|^2dx-\frac{\lambda}{2m}\int_\Omega e^{mu^2}dx$$
(where $\lambda=\lambda_k$ plays the role of a Lagrange multiplier), which is well defined and smooth thanks to \eqref{ada}, but does not satisfy the Palais-Smale condition. For a more detailed discussion, in the context of Orlicz spaces, we refer to \cite{str}.

\medskip

The function $\eta_0$ which appears in \eqref{etak} is a solution of the higher-order Liouville's 
equation
\begin{equation}\label{eqlio}
(-\Delta)^m \eta_0 =(2m-1)!e^{2m\eta_0},\quad \text{on }\R{2m}.
\end{equation}
We recall (see e.g. \cite{mar1}) that if $u$ solves $(-\Delta)^m u=Ve^{2mu}$ on $\R{2m}$, then the conformal metric $g_u:=e^{2u}g_{\R{2m}}$ has $Q$-curvature $V$, where $g_{\R{2m}}$ denotes the Euclidean metric.
This shows a surprising relation between Equation \eqref{eq0} and the problem of prescribing the $Q$-curvature. In fact $\eta_0$ has also a remarkable geometric interpretation: If $\pi:S^{2m}\to \R{2m}$ is the stereographic projection, then
\begin{equation}\label{proj}
e^{2\eta_0}g_{\R{2m}}=(\pi^{-1})^* g_{S^{2m}},
\end{equation}
where $g_{S^{2m}}$ is the round metric on $S^{2m}$.
Then \eqref{proj} implies
\begin{equation}\label{etak2}
(2m-1)!\int_{\R{2m}}e^{2m \eta_0}dx=\int_{S^{2m}}Q_{S^{2m}}\mathrm{dvol}_{g_{S^{2m}}}=(2m-1)!|S^{2m}|=\Lambda_1.
\end{equation}
This is the reason why $\Lambda\geq I\Lambda_1$ in case (ii) of Theorem \ref{trm1} above, compare Proposition \ref{blowup}.

\medskip

Theorem \ref{trm1} have been proved by Adimurthi and M. Struwe \cite{AS} and Adimurthi and O. Druet \cite{AD} in the case $m=1$, and by F. Robert and M. Struwe \cite{RS} for $m=2$, and
we refer to them for further motivations and references. Here, instead, we want to point out the main ingredients of our approach. Crucial to the proof of Theorem \ref{trm1} are the gradient estimates in Lemma \ref{stimaBR} and the blow-up procedure of Proposition \ref{blowup}. For the latter, we rely on a concentration-compactness result from \cite{mar2} and a classification result from \cite{mar1}, which imply, together with the gradient estimates, that at the finitely many concentration points $\{x^{(1)},\ldots, x^{(I)}\}$, the profile of $u_k$ is $\eta_0$, hence an energy not less that $\Lambda_1$ accumulates, namely
$$\lim_{R\to 0}\limsup_{k\to\infty}\int_{B_R(x^{(i)})}\lambda_k u_k^2 e^{mu_k^2}dx\geq \Lambda_1,\quad \text{for every }1\leq i\leq I.$$
As for the gradient estimates, if one uses \eqref{eq0} and \eqref{eq1} to infer $\|\Delta^m u_k\|_{L^1(\Omega)}\leq C$, then elliptic regularity gives $\|\nabla^\ell u_k\|_{L^p(\Omega)}\leq C(p)$ for every $p\in [1,2m/\ell)$. These bounds, though, turn out to be too weak for Lemma \ref{stimaBR} (see also the remark after Lemma \ref{lorentz}). One has, instead, to fully exploit the integrability of $\Delta^m u_k$ given by \eqref{eq1}, namely $\|\Delta^m u_k\|_{L(\log L)^{1/2}(\Omega)}\leq C$, where $L(\log L)^{1/2}\subsetneq L^1$ is the Zygmund space. Then an interpolation result from \cite{BS} gives uniform estimates for $\nabla^\ell u_k$ in the Lorentz space $L^{(2m/\ell, 2)}(\Omega)$, $1\leq \ell\leq 2m-1$, which are sharp for our purposes (see Lemma \ref{lorentz}).

We remark that when $m=1$, things simplify dramatically, as we can simply integrate by parts \eqref{eq1} and get
$$\|\nabla u_k\|_{L^{(2,2)}(\Omega)}=\|\nabla u_k\|_{L^2(\Omega)}\leq C.$$
In the case $m=2$, F. Robert and M. Struwe \cite{RS} proved a slightly weaker form of our Lemma \ref{stimaBR} by using subtle estimates in the $BMO$ space, whose generalization to arbitrary dimensions appears quite challenging. Our approach, on the other hand, is simpler and more transparent.

\medskip

Recently O. Druet \cite{dru} for the case $m=1$, and M. Struwe \cite{str2} for $m=2$ improved the previous results by showing that in case (ii) of Theorem \ref{trm1} we have $\Lambda=L\Lambda_1$ for some positive $L\in \mathbb{N}$.

\medskip

In the following, the letter $C$ denotes a generic positive constant, which may change from line to line and even within the same line.

\medskip

I'm grateful to Prof. Michael Struwe for many useful discussions.

\section{Proof of Theorem \ref{trm1}}\label{quadratic}

Assume first that $\sup_\Omega u_k\leq C$. Then $\Delta^m u_k\to 0$ uniformly, since $\lambda_k\to 0$. By elliptic estimates we infer $u_k\to 0$ in $W^{2m,p}(\Omega)$ for every $1\leq p<\infty$, hence $u_k\to 0$ in $C^{2m-1,\alpha}(\Omega)$, $\Lambda=0$ and we are in case (i) of Theorem \ref{trm1}.

From now on, following the approach of \cite{RS}, we assume that, up to a subsequence, $\sup_\Omega u_k\to\infty$ and show that we are in case (ii) of the theorem. In Section \ref{blowup1} we analyze the asymptotic profile at blow-up points. In Section \ref{blowup2} we sketch the inductive procedure which completes the proof.

\subsection{Analysis of the first blow-up}\label{blowup1}

Let $x_k=x_{1,k}\in\Omega$ be a point such that
$u_k(x_k)=\max_{\Omega}u_k,$
and let $r_k=r_{1,k}$ be as in \eqref{rik}.
Integrating by parts in \eqref{eq1}, we find $\|\Delta^\frac{m}{2}u_k\|_{L^2(\Omega)}\leq C$ which, together with the boundary condition and elliptic estimates (see e.g. \cite{ADN}), gives
\begin{equation}\label{Hm}
\|u_k\|_{H^m(\Omega)}\leq C.
\end{equation}

\begin{lemma}\label{boundary} We have
$$\lim_{k\to\infty} \frac{\dist(x_k,\partial\Omega)}{r_k}=+\infty.$$
\end{lemma}

\begin{proof} Set 
$$\overline u_k(x):= \frac{u_k(r_kx+x_k)}{u_k(x_k)} \quad \textrm{for }x\in \Omega_k:=\{r_k^{-1}(x-x_k):x\in\Omega\}.$$
Then $\overline u_k$ satisfies
$$
\left\{
\begin{array}{ll}
\displaystyle (-\Delta)^m \overline u_k=\frac{2^{2m}(2m-1)!}{u_k^2(x_k)}\overline u_k e^{m u_k^2(x_k)(\overline u_k^2-1)} & \textrm{in }\Omega_k\\
\displaystyle \overline u_k>0 & \textrm{in }\Omega_k\\
\displaystyle \overline u_k=\partial_\nu \overline u_k=\ldots =\partial_\nu^{m-1}\overline u_k =0& \textrm{on }\partial \Omega_k.
\end{array}
\right.
$$
Assume for the sake of contradiction that up to a subsequence we have $$\lim_{k\to\infty}\frac{\dist(x_k,\partial\Omega)}{r_k}=R_0<+\infty.$$
Then, passing to a further subsequence, $\Omega_k\to \mathcal{P}$, where $\mathcal{P}$ is a half-space. Since
$$\|\Delta^m \overline u_k\|_{L^\infty(\Omega_k)}\leq \frac{C}{u_k^2(x_k)}\to 0\quad \text{as }k\to\infty,$$
we see that, up to a subsequence, $\overline u_k\to\overline u$ in $C^{2m-1,\alpha}_{\loc}(\overline{\mathcal{P}})$, where
$$\overline u(0)=\overline u_k(0)=1$$
and
$$
\left\{
\begin{array}{ll}
\displaystyle (-\Delta)^m \overline u=0 & \textrm{in }\mathcal{P}\\
\displaystyle \overline u=\partial_\nu \overline u=\ldots =\partial_\nu^{m-1}\overline u=0 & \textrm{on }\partial \mathcal{P}.
\end{array}
\right.
$$
By \eqref{Hm} and the Sobolev imbedding $H^{m-1}(\Omega)\hookrightarrow L^{2m}(\Omega)$, we find
$$\int_{\Omega_k} |\nabla \overline u_k|^{2m}dx =\frac{1}{u_k(x_k)^{2m}}\int_\Omega|\nabla u_k|^{2m}dx\leq \frac{C}{u_k(x_k)^{2m}}\to 0, \quad \text{as }k\to\infty.$$
Then $\nabla \overline u\equiv 0$, hence $\overline u\equiv const=0$ thanks to the boundary condition. That contradicts $\overline u(0)=1$.
\end{proof}

\begin{lemma}\label{lemmauk} We have
\begin{equation}\label{wk}
u_k(x_k+r_kx)-u_k(x_k)\to 0 \quad \textrm{ in }C^{2m-1}_{\loc}(\R{2m})\textrm{ as } k\to\infty.
\end{equation}
\end{lemma}

\begin{proof} Set
$$ v_k(x):=u_k(x_k+r_k x)-u_k(x_k),\quad x\in \Omega_k $$
Then $v_k$ solves
\begin{equation}\label{eqwk}
(-\Delta)^m v_k =2^{2m}(2m-1)!\frac{\overline u_k(x)}{u_k(x_k)}e^{mu_k^2(x_k)(\overline u_k^2-1)}\leq 2^{2m}\frac{(2m-1)!}{u_k(x_k)}\to 0.
\end{equation}
Assume that $m>1$. By \eqref{Hm} and the Sobolev embedding $H^{m-2}(\Omega)\hookrightarrow L^m(\Omega)$, we get
\begin{equation}\label{D2vk}
\|\nabla^2 v_k\|_{L^m(\Omega_k)}=\|\nabla^2 u_k\|_{L^m(\Omega)}\leq C.
\end{equation}
Fix now $R>0$ and write $v_k=h_k+w_k$ on $B_R=B_R(0)$, where $\Delta^m h_k=0$ and $w_k$ satisfies the Navier-boundary condition on $B_R$. Then, \eqref{eqwk} gives

\begin{equation}\label{wk0}
w_k\to 0 \quad \textrm{in } C^{2m-1,\alpha}(B_R).
\end{equation}
This, together with \eqref{D2vk} implies
\begin{equation}\label{hk0}
\|\Delta h_k\|_{L^m(B_R)}\leq C.
\end{equation}
Then, since $\Delta^{m-1}(\Delta h_k)=0$, we get from Proposition \ref{c2m}
\begin{equation}\label{hk2}
\|\Delta h_k\|_{C^\ell(B_{R/2})}\leq C(\ell) \quad\textrm{for every }\ell\in\mathbb{N}.
\end{equation}
By Pizzetti's formula \eqref{pizzetti},
$$\Intm_{B_R}h_k dx=h_k(0)+\sum_{i=1}^{m-1}c_iR^{2i}\Delta^i h_k(0),$$
and \eqref{hk2}, together with $|h_k(0)|=|w_k(0)|\leq C$ and $h_k\leq -w_k\leq C$, we find
$$\Intm_{B_R}|h_k|dx\leq C.$$
Again by Proposition \ref{c2m} it follows that
\begin{equation}\label{hk3}
\|h_k\|_{C^\ell(B_{R/2})}\leq C(\ell) \quad\textrm{for every }\ell\in\mathbb{N}.
\end{equation}
By Ascoli-Arzel\`a's theorem, \eqref{wk0} and \eqref{hk3}, we have that up to a subsequence
$$v_k\to v\quad\textrm{in }C^{2m-1,\alpha}(B_{R/2}),$$
where $\Delta^m v\equiv 0$ thanks to \eqref{eqwk}.
We can now apply the above procedure with a sequence of radii $R_k\to \infty$, extract a diagonal subsequence $(v_{k'})$, and find a function $v\in C^\infty(\R{2m})$ such that
\begin{equation}\label{55'}
v\leq 0,\quad \Delta^m v\equiv 0, \quad v_{k'}\to v\quad\textrm{in }C^{2m-1,\alpha}_{\loc}(\R{2m}).
\end{equation}
By Fatou's Lemma
\begin{equation}\label{55*}
\|\nabla^2 v\|_{L^m(\R{2m})}\leq \liminf_{k\to\infty}\|\nabla^2 v_{k'}\|_{ L^m(\Omega_k)}\leq C.
\end{equation}
By Theorem \ref{trmliou} and \eqref{55'}, $v$ is a polynomial of degree at most $2m-2$. Then \eqref{55'} and \eqref{55*} imply that $v$ is constant, hence $v\equiv v(0)=0$. Therefore the limit does not depend on the chosen subsequence $(v_{k'})$, and the full sequence $(v_k)$ converges to $0$ in $C^{2m-1}_{\loc}(\R{2m})$, as claimed.

When $m=1$, Pizzetti's formula and \eqref{eqwk} imply at once that, for every $R>0$, $\|v_k\|_{L^1(B_R)}\to 0$, hence  $v_k\to 0$ in $W^{2,p}(B_{R/2})$ as $k\to \infty$, $1\leq p<\infty$. 
\end{proof}

Now set
\begin{equation}\label{defetak}
\eta_k(x):=u_k(x_k)[u_k(r_kx +x_k)-u_k(x_k)]+\log 2\leq \log 2.
\end{equation}
An immediate consequence of
Lemma \ref{lemmauk} is the following

\begin{cor}\label{coretak} The function $\eta_k$ satisfies
\begin{equation}\label{deltaeta}
(-\Delta)^m \eta_k =V_k e^{2m a_k \eta_k},
\end{equation}
where
$$V_k(x)= 2^{m(1-\overline u_k)}(2m-1)! \overline u_k(x)\to (2m-1)!,\quad a_k=\frac{1}{2}(\overline u_k+1)\to 1$$
in $C^0_{\loc}(\R{2m})$.
\end{cor}

\begin{lemma}\label{lorentz} For every $1\leq \ell\leq 2m-1$, $\nabla^\ell u_k$ belongs to the Lorentz space $L^{(2m/\ell,2)}(\Omega)$ and
\begin{equation}\label{stimalorentz}
\|\nabla^\ell u_k\|_{(2m/\ell,2)}\leq C.
\end{equation}
\end{lemma}

\begin{proof} We first show that $f_k:=(-\Delta)^m u_k$ is bounded in $L (\log L)^\frac{1}{2}(\Omega)$, where
$$L (\log L)^\alpha(\Omega):=\bigg\{f\in L^1(\Omega):\|f\|_{L(\log L)^\alpha}:=\int_\Omega |f|\log^\alpha (2+|f|)dx<\infty \bigg\}.$$
Indeed, set $\log^+ t:=\max\{0,\log t\}$ for $t>0$. Then, using the simple inequalities
$$\log(2+t)\leq 2+\log^+ t,\quad \log^+(ts)\leq \log^+ t+\log^+ s,\quad t,s>0,$$
one gets
$$\log(2+\lambda_ku_k e^{mu_k^2})\leq 2+\log^+\lambda_k+\log^+ u_k+mu_k^2\leq C(1+u_k)^2.$$
Then, since $f_k\geq 0$, we have
\begin{eqnarray*}
\|f_k\|_{L(\log L)^\frac{1}{2}}&\leq&\int_\Omega f_k\log^\frac{1}{2}(2+f_k)dx\\
&\leq& C\int_{\{x\in \Omega:u_k(x)\geq 1\}} \lambda_k u_k^2 e^{mu_k}dx +C|\Omega|\leq C
\end{eqnarray*}
by \eqref{eq1}, as claimed. Now \eqref{stimalorentz} follows from Theorem \ref{LlogL}. 
\end{proof}

\begin{rmk} The inequality \eqref{stimalorentz} is intermediate between the $L^1$ and the $L \log L$ estimates.
Indeed, the bound of $f_k:=(-\Delta)^m u_k$ in $L^1$ implies $\|\nabla^\ell u_k\|_{L^p}\leq C$ for every $1\leq \ell\leq 2m-1$, $1\leq p <\frac{2m}{\ell}$, and actually $\|\nabla^\ell u_k\|_{(2m/\ell,\infty)}\leq C$ (compare \cite[Thm. 3.3.6]{Hel}), but that is not enough for our purposes (Lemma \ref{stimaBR} below). On the other hand, was $f_k$ bounded in $L (\log L)$, we would have $\|\nabla^\ell u_k\|_{(2m/\ell,1)}\leq C$, which implies $\|u_k\|_{L^\infty}\leq C$ (compare \cite[Thm. 3.3.8]{Hel}). But we know that this is not the case in general.

Actually, the cases $1\leq \ell\leq m$ in \eqref{stimalorentz} follow already from \eqref{Hm} and the improved Sobolev embeddings, see \cite{ON}. What really matters here are the cases $m< \ell<2m$. In fact, when $m=1$ Lemma \ref{lorentz} reduces to \eqref{Hm}.
\end{rmk}

The following lemma replaces and sharpens Proposition 2.3 in \cite{RS}.

\begin{lemma}\label{stimaBR} For any $R>0$, $1\leq \ell\leq 2m-1$ there exists $k_0=k_0(R)$ such that
$$u_k(x_k)\int_{B_{Rr_k}(x_k)}|\nabla^\ell u_k|dx\leq C (Rr_k)^{2m-\ell},\quad \textrm{for all }k\geq k_0. $$
\end{lemma}

\begin{proof} We first claim that
\begin{equation}\label{Deltau2}
\|\Delta^m(u_k^2)\|_{L^1(\Omega)}\leq C.
\end{equation}
To see that, observe that
\begin{equation}\label{Deltau3}
|\Delta^m (u_k^2)|\leq 2u_k (-\Delta)^m u_k+C\sum_{\ell=1}^{2m-1}|\nabla^\ell u_k||\nabla^{2m-\ell}u_k|.
\end{equation}
The term $2u_k(-\Delta)^m u_k$ is bounded in $L^1$ thanks to  \eqref{eq1}. The other terms on the right-hand side of \eqref{Deltau3} are bounded in $L^1$ thanks to Lemma \ref{lorentz} and the H\"older-type inequality of O'Neil \cite{ON}.\footnote{If $\frac{1}{p}+\frac{1}{p'}=\frac{1}{q}+\frac{1}{q'}=1$, and $f\in L^{(p,q)}$, $g\in L^{(p',q')}$, then $\|fg\|_{L^1}\leq \|f\|_{(p,q)}\|g\|_{(p',q')}$.} Hence \eqref{Deltau2} is proven.

Now set $f_k:=(-\Delta)^m(u_k^2)$, and for any $x\in \Omega$, let $G_x$ be the Green's function for $(-\Delta)^m$ on $\Omega$ with Dirichlet boundary condition.
Then
$$u_k^2(x)=\int_{\Omega}G_x(y)f_k(y)dy.$$
Thanks to \cite[Thm. 12]{DAS}, $|\nabla^\ell G_x(y)|\leq C|x-y|^{-\ell}$, hence
$$|\nabla^\ell (u_k^2)(x)|\leq\int_\Omega |\nabla_x^\ell G_x(y)||f_k(y)|dy\leq C\int_{\Omega}\frac{|f_k(y)|}{|x-y|^\ell}dy.$$
Let $\mu_k$ denote the probability measure $\frac{|f_k(y)|}{\|f_k\|_{L^1(\Omega)}}dy$. By Fubini's theorem
\begin{eqnarray*}
\int_{B_{Rr_k}(x_k)}|\nabla^\ell (u_k^2)(x)|dx&\leq &C\|f_k\|_{L^1(\Omega)}\int_{B_{Rr_k}(x_k)} \int_{\Omega}\frac{1}{|x-y|^\ell}d\mu_k(y)dx\\
&\leq&C \int_\Omega  \int_{B_{Rr_k}(x_k)}\frac{1}{|x-y|^\ell}dxd\mu_k(y)\\
&\leq& C\sup_{y\in\Omega} \int_{B_{Rr_k}(x_k)}\frac{1}{|x-y|^\ell}dx
\leq C  (Rr_k)^{2m-\ell}. 
\end{eqnarray*}
To conclude the proof, observe that Lemma \ref{lemmauk} implies that on $B_{Rr_k}(x_k)$, for $1\leq \ell\leq 2m-1$, we have $r_k^\ell\nabla^\ell u_k\to 0$ uniformly, hence  
\begin{eqnarray*}
u_k(x_k)|\nabla^\ell u_k|&\leq& C u_k|\nabla^\ell u_k|\leq C\bigg(|\nabla^\ell(u_k^2)|+\sum_{j=1}^{\ell-1}|\nabla^j u_k||\nabla^{\ell-j}u_k|\bigg)\\
&\leq& C |\nabla^\ell(u_k^2)| +o(r_k^{-\ell}),\quad \textrm{as }k\to\infty.
\end{eqnarray*}
Integrating over $B_{Rr_k}(x_k)$ and using the above estimates we conclude.
\end{proof}

\begin{prop}\label{blowup} Let $\eta_k$ be as in \eqref{defetak}. Then, up to selecting a subsequence,
$\eta_k(x)\to \eta_0(x)=\log\frac{2}{1+|x|^2}$ in $C^{2m-1}_{\loc}(\R{2m})$, and
\begin{equation}\label{limen}
\lim_{R\to\infty}\lim_{k\to\infty}\int_{B_{Rr_k}(x_k)}\lambda_k u_k^2 e^{mu_k^2}dx=\lim_{R\to\infty} (2m-1)!\int_{B_R(0)}e^{2m \eta_0}dx=\Lambda_1.
\end{equation}
\end{prop}

\begin{proof} Fix $R>0$, and notice that, thanks to Lemma \ref{lemmauk} and \eqref{deltaeta},
\begin{eqnarray}
\int_{B_R(0)}V_k e^{2m a_k\eta_k}dx &=& 
\int_{B_{Rr_k}(x_k)} u_k(x_k) u_k \lambda_k e^{m u_k^2}dx\label{eneta}\\
&\leq&(1+o(1))\int_{B_{Rr_k}(x_k)} u_k^2 \lambda_k e^{m u_k^2}dx \leq \Lambda +o(1) \nonumber,
\end{eqnarray}
where $V_k$ and $a_k$ are as in Corollary \ref{coretak}, and $o(1)\to 0$ as $k\to\infty$.

\medskip

\noindent\emph{Step 1.} We claim that $\eta_k\to \overline \eta$ in $C^{2m-1}_{\loc}(\R{2m})$, where $\overline \eta$ satisfies
\begin{equation}\label{deltaeta0}
(-\Delta)^m \overline \eta=(2m-1)!e^{2m \overline\eta}.
\end{equation}
Then, letting $R\to\infty$ in \eqref{eneta}, from Corollary \ref{coretak} and Fatou's lemma we infer $e^{2m \overline\eta}\in L^1(\R{2m}).$

Let us prove the claim. Consider first the case $m>1$. From Corollary \ref{coretak}, Theorem 1 in \cite{mar2}, and \eqref{eneta}, together with $\eta_k\leq \log 2$ (which implies that $S_1=\emptyset$ in Theorem 1 of \cite{mar2}), we infer that up to subsequences either
\begin{itemize}
\item[(i)] $\eta_k\to \overline\eta$ in $C^{2m-1}_{\loc}(\R{2m})$ for some function $\overline \eta\in C^{2m-1}_{\loc}(\R{2m})$, or

\item[(ii)] $\eta_k\to -\infty$ locally uniformly in $\R{2m}$, or

\item[(iii)] there exists a closed set $S_0\neq\emptyset$ of Hausdorff dimension at most $2m-1$ and numbers $\beta_k\to+\infty$ such that
$$\frac{\eta_k}{\beta_k}\to\varphi \textrm{ in } C^{2m-1}_{\loc}(\R{2m}\backslash S_0),$$
where
\begin{equation}\label{varphi}
\Delta^m\varphi \equiv 0,\quad \varphi\leq 0,\quad \varphi\not\equiv 0 \quad \textrm{on }\R{2m},\quad \varphi\equiv 0\textrm{ on } S_0.
\end{equation}
\end{itemize}
Since $\eta_k(0)=\log 2$, (ii) can be ruled out. Assume now that (iii) occurs. From Liouville's theorem and \eqref{varphi} we get $\Delta \varphi\not\equiv 0$, hence for some $R>0$ we have
$\int_{B_R}|\Delta \varphi|dx>0$ and
\begin{equation}\label{deltainf}
\lim_{k\to\infty}\int_{B_R}|\Delta \eta_k|dx=\lim_{k\to\infty}\beta_k\int_{B_R}|\Delta \varphi|dx=+\infty.
\end{equation}
On the other hand, we infer from Lemma \ref{stimaBR}
\begin{equation}\label{nablaell2}
\int_{B_R}|\nabla^\ell\eta_k|dx= u_k(x_k)r_k^{\ell-2m}\int_{B_{Rr_k}(x_k)}|\nabla^\ell u_k|dx\leq C R^{2m-\ell},
\end{equation}
contradicting \eqref{deltainf} when $\ell=2$ and therefore proving our claim.

When $m=1$, Theorem 3 in \cite{BM} implies that only Case (i) or Case (ii) above can occur. Again Case (ii) can be ruled out, since $\eta_k(0)=\log 2$, and we are done. 

\medskip

\noindent\emph{Step 2.}
We now prove that $\overline\eta$ is a standard solution of \eqref{deltaeta0}, i.e. there are $\lambda>0$ and $x_0\in \R{2m}$ such that
\begin{equation}\label{standard}
\overline\eta(x)=\log\frac{2\lambda}{1+\lambda^2|x-x_0|^2}.
\end{equation}
For $m=1$ this follows at once from \cite{CL}. For $m>1$, if $\overline\eta$ didn't have the form \eqref{standard}, according to \cite[Thm. 2]{mar1} (see also \cite{lin} for the case $m=2$), there would exist $j\in \mathbb{N}$ with $1\leq j\leq m-1$, and $a<0$ such that
\begin{equation*}
\lim_{|x|\to\infty}(-\Delta)^j \overline\eta(x)=a.
\end{equation*} 
This would imply
$$\lim_{k\to\infty} \int_{B_R(0)}|\Delta^j\eta_k|dx=|a|\cdot \mathrm{vol}(B_1(0))R^{2m}+o(R^{2m}) \quad \textrm{as }R\to\infty,$$
contradicting \eqref{nablaell2} for $\ell=2j$. Hence \eqref{standard} is established. Since $\eta_k\leq \eta_k(0)=\log 2$, it follows immediately that $x_0=0$, $\lambda=1$, i.e. $\overline\eta=\eta_0$, and \eqref{limen} follows from \eqref{etak2}, \eqref{eneta} and Fatou's lemma.
\end{proof}

\subsection{Exhaustion of the blow-up points and proof of Theorem \ref{trm1}}\label{blowup2}

For $\ell\in\mathbb{N}$ we say that $(H_\ell)$ holds if there are $\ell$ sequences of converging points $x_{i,k}\to x^{(i)}$, $1\leq i\leq \ell$ such that
\begin{equation}\label{Hl}
\sup_{x\in \Omega}\lambda_k R_{\ell,k}^{2m}(x)u_k^2(x)e^{mu_k^2(x)}\leq C,
\end{equation}
where
$$R_{\ell,k}(x):=\inf_{1\leq i\leq \ell}|x-x_{i,k}|.$$
We say that $(E_\ell)$ holds if there are $\ell$ sequences of converging points $x_{i,k}\to x^{(i)}$ such that, if we define $r_{i,k}$ as in \eqref{rik}, the following hold true:

\begin{itemize}
\item[$(E_\ell^1)$] For all $1\leq i\neq j\leq \ell$
$$\lim_{k\to\infty}\frac{\dist(x_{i,k},\partial\Omega)}{r_{i,k}}=\infty,\qquad \lim_{k\to\infty}\frac{|x_{i,k}-x_{j,k}|}{r_{i,k}}=\infty.$$
\item[$(E_\ell^2)$] For $1\leq i\leq \ell$ \eqref{etak} holds true.
\item[$(E_\ell^3)$] $\lim_{R\to\infty}\lim_{k\to\infty}\int_{\cup_{i=1}^\ell B_{Rr_{i,k}}(x_{i,k})}\lambda_k u_k^2e^{mu_k^2}dx=\ell\Lambda_1.$
\end{itemize}

To prove Theorem \ref{trm1} we show inductively that $(H_I)$ and $(E_I)$ hold for some positive $I\in \mathbb{N}$ (with the same sequences $x_{i,k}\to x^{(i)}$, $1\leq i\leq I$), following the approach of \cite{AD} and \cite{RS}. First observe that $(E_1)$ holds thanks to Lemma \ref{boundary} and Proposition \ref{blowup}. Assume now that for some $\ell\geq 1$ $(E_\ell)$ holds and $(H_{\ell})$ does not. Choose $x_{\ell+1,k}\in \Omega$ such that
\begin{equation}\label{supinf}
\lambda_k R_{\ell,k}^{2m}(x_{\ell+1,k})u_k^2(x_{\ell+1,k})e^{mu_k^2(x_{\ell+1,k})}=\lambda_k\max_{\Omega} R_{\ell,k}^{2m}u_k^2 e^{m u_k^2}\to\infty\quad \text{as } k\to\infty
\end{equation}
and define $r_{\ell+1,k}$ as in \eqref{rik}. It easily follows from \eqref{supinf} that
\begin{equation}\label{yk}
\lim_{k\to\infty}\frac{|x_{\ell+1,k}-x_{i,k}|}{r_{\ell+1,k}}=\infty,\quad 1\leq i\leq \ell.
\end{equation}
Moreover, thanks to $(E_{\ell}^2)$ and \eqref{supinf}, we also have
$$\lim_{k\to\infty}\frac{|x_{\ell+1,k}-x_{i,k}|}{r_{i,k}}=\infty\quad \text{for }1\leq i\leq \ell.$$
We now need to replace Lemma \ref{boundary} and Lemma \ref{lemmauk} with the lemma below.

\begin{lemma}\label{lemmauk1} Under the above assumptions and notation, we have
\begin{equation}\label{distxl1}
\lim_{k\to\infty} \frac{\dist (x_{\ell+1,k},\partial \Omega)}{r_{\ell+1,k}}=\infty
\end{equation}
and
\begin{equation}\label{el3}
u_k(x_{\ell+1,k}+r_{\ell+1,k}x)-u_k(x_{\ell+1,k})\to 0 \quad\textrm{in }C^{2m-1}_{\loc}(\R{2m}),\quad \textrm{as }k\to\infty.
\end{equation}
\end{lemma}

\begin{proof} To simplify the notation, let us write $y_k:=x_{\ell+1,k}$ and $\rho_k:= r_{\ell+1,k}$.
Evaluating the right-hand side of \eqref{supinf} at the point $y_k+\rho_k x$ we get
\begin{eqnarray*}
&&\Big(\inf_{1\leq i\leq \ell} |y_k-x_{i,k}+\rho_k x|^{2m}\Big)u_k^2(y_k+\rho_k x)e^{mu_k^2(y_k+\rho_k x)}\\
&&\leq
\Big(\inf_{1\leq i\leq \ell} |y_k-x_{i,k}|^{2m}\Big) u_k^2(y_k)e^{mu_k^2(y_k)},
\end{eqnarray*}
Hence, setting $\overline u_{\ell+1,k}(x):=\frac{u_k(y_k+\rho_k x)}{u_k(y_k)}$, we have that
\begin{equation}\label{el1}
\overline u_{\ell+1,k}^2(x)e^{m u_k^2(y_k)(\overline u^2_{\ell+1,k}(x)-1)}\leq\frac{\inf_{1\leq i\leq \ell}|y_k-x_{i,k}|^{2m}}{\inf_{1\leq i\leq \ell}|y_k-x_{i,k}+\rho_k x|^{2m}}= 1+o(1),
\end{equation}
where $o(1)\to 0$ as $k\to\infty$ locally uniformly in $x$, as \eqref{yk} immediately implies. 
Then \eqref{distxl1} follows as in the proof of Lemma \ref{boundary}, since \eqref{el1} implies
\begin{equation}\label{el2}
(-\Delta)^m\overline u_{\ell+1,k}=\frac{2^{2m}(2m-1)!}{u_k^2(y_k)}\overline u_{\ell+1,k}e^{m u_k^2(y_k)(\overline u^2_{\ell+1,k}-1)}=o(1),
\end{equation}
where $o(1)\to 0$ as $k\to \infty$ uniformly locally in $\R{2m}$.

\medskip

Define now $v_k(x):=u_k(x_{\ell+1,k}+r_{\ell+1,k}x)-u_k(x_{\ell+1,k})$, and observe that
$$u_k(y_k+\rho_k x)\to\infty \quad \text{locally uniformly in }\R{2m}, $$
thanks to \eqref{supinf} and \eqref{yk}. This and \eqref{el2} imply that we can replace \eqref{eqwk} in the proof of Lemma \ref{lemmauk} with
$$(-\Delta)^mv_k= 2^{2m}(2m-1)!\frac{\overline u_k^2}{u_k(y_k+\rho_k \cdot)}e^{mu_k^2(y_k)(\overline u_{\ell+1,k}^2-1)}\to 0\quad \text{in } L^\infty_{\loc}(\R{2m}).$$
Then the rest of the proof of Lemma \ref{lemmauk} applies without changes, and also \eqref{el3} is proved.
\end{proof}

Still repeating the arguments of the preceding section with $x_{\ell+1,k}$ instead of $x_k$ and $r_{\ell+1,k}$ instead of $r_k$, we define
$$\eta_{\ell+1,k}(x):=u_k(x_{\ell+1,k})[u_k(r_{\ell+1,k}x+x_{\ell+1,k})-u_k(x_{\ell+1,k})],$$
and we have

\begin{prop}\label{blowup3} Up to a subsequence
$$\eta_{\ell+1,k}(x)\to \eta_0(x)=\log\frac{2}{1+|x|^2}\quad \text{in } C^{2m-1,\alpha}_{\loc}(\R{2m})$$
and
\begin{equation}\label{limen2}
\lim_{R\to\infty}\lim_{k\to\infty}\int_{B_{R r_{\ell+1,k}}(x_{\ell+1,k})}\lambda_k u_k^2 e^{mu_k^2}dx=\lim_{R\to\infty}\int_{B_R(0)}e^{2m\eta_0}dx=\Lambda_1.
\end{equation}
\end{prop}

Summarizing, we have proved that $(E^1_{\ell+1})$, $(E^2_{\ell+1})$ and \eqref{limen2} hold. These also imply that $(E^3_{\ell+1})$ holds, hence we have $(E_{\ell+1})$.
Because of \eqref{eq1} and $(E_\ell^3)$, the procedure stops in a finite number $I$ of steps, and we have $(H_I)$.

\medskip

Finally, we claim that $\lambda_k\to 0$ implies $u_k\rightharpoonup 0$ in $H^m(\Omega)$. This, \eqref{Rk} and elliptic estimates then imply that
$$u_k\to 0\quad \textrm{in} \quad C^{2m-1,\alpha}_{\loc}(\Omega\backslash \{x^{(1)},\ldots,x^{(I)}\}).$$
To prove the claim, we observe that for any $\alpha>0$
\begin{eqnarray*}
\int_{\Omega} |\Delta^m u_k| dx&=&\int_{\Omega}\lambda_k u_k e^{m u_k^2}dx\\
&\leq& \frac{\lambda_k}{\alpha}\int_{\{x\in\Omega: u_k\geq \alpha\}}u_k^2 e^{mu_k^2}dx+\lambda_k \int_{\{x\in\Omega: u_k< \alpha\}}u_ke^{mu_k^2}dx\\
&\leq&\frac{C}{\alpha}+\lambda_k C_\alpha,
\end{eqnarray*}
where $C_\alpha$ depends only on $\alpha$. Letting $k$ and $\alpha$ go to infinity, we infer
\begin{equation}\label{delta0}
\Delta^m u_k\to 0\quad\textrm{in }L^1(\Omega).
\end{equation}
Thanks to \eqref{Hm}, we infer that up to a subsequence $u_k\rightharpoonup u_0$ in $H^m(\Omega)$. Then \eqref{delta0} and the boundary condition imply that $u_0\equiv 0$, in particular the full sequence converges to $0$ weakly in $H^m(\Omega)$. This completes the proof of the theorem.

\section*{Appendix}

\appendix

\subsection*{An elliptic estimate for Zygmund and Lorentz spaces}

\begin{trm}\label{LlogL} Let $u$ solve $\Delta^m u=f\in L(\log L)^\alpha$ in $\Omega$ with Dirichlet boundary condition, $0\leq \alpha\leq 1$, $\Omega\subset\R{n}$ bounded and with smooth boundary, $n\geq 2m$. Then $\nabla^{2m-\ell} u\in L^{\big(\frac{n}{n-\ell},\frac{1}{\alpha}\big)}(\Omega)$, $1\leq \ell\leq 2m-1$ and
\begin{equation}\label{LlogL1}
\|\nabla^{2m-\ell} u\|_{\big(\frac{n}{n-\ell},\frac{1}{\alpha}\big)}\leq C \|f\|_{L(\log L)^\alpha}.
\end{equation}
\end{trm}

\begin{proof} Define
$$\hat f:=
\left\{
\begin{array}{ll}
f& \textrm{in }\Omega \\
0& \textrm{in }\R{n}\backslash\Omega,
\end{array}
\right.
$$
and let $w:=K*\hat f$, where $K$ is the fundamental solution of $\Delta^m$. Then
$$|\nabla^{2m-1}w|=|(\nabla^{2m-1}K)*\hat f|\leq CI_1*|\hat f|,$$
where $I_1(x)=|x|^{1-n}$.
According to \cite[Cor. 6.16]{BS}, $|\nabla^{2m-1}w|\in L^{\big(\frac{n}{n-1},\frac{1}{\alpha}\big)}(\R{n})$ and
\begin{equation}\label{LlogL2}
\|\nabla^{2m-1} w\|_{\big(\frac{n}{n-1},\frac{1}{\alpha}\big)}\leq C\|\hat f\|_{L(\log L)^\alpha}=C\|f\|_{L(\log L)^\alpha}.
\end{equation}
We now use \eqref{LlogL2} to prove \eqref{LlogL1}, following a method that we learned from \cite{Hel}. Given $g:\Omega\to\R{n}$ measurable, let
$v_g$ be the solution to $\Delta^m v_g=\diver g$ in $\Omega$, with the same boundary condition as $u$, and
set $P(g):=|\nabla^{2m-1}v_g|$.
By $L^p$ estimates (see e.g. \cite{ADN}), $P$ is bounded from $L^p(\Omega;\R{n})$ into $L^p(\Omega)$ for $1<p<\infty$. Then, thanks to the interpolation theory for Lorentz spaces, see e.g. \cite[Thm. 3.3.3]{Hel}, $P$ is bounded from $L^{(p,q)}(\Omega;\R{n})$ into $L^{(p,q)}(\Omega)$ for $1<p<\infty$ and $1\leq q\leq \infty$. 
Choosing now $g=\nabla\Delta^{m-1}w$, we get $v_g=u$, hence 
$|\nabla^{2m-1}u|=P(\nabla\Delta^{m-1}w)$, and from \eqref{LlogL2} we infer
$$\|\nabla^{2m-1}u\|_{\big(\frac{n}{n-1},\frac{1}{\alpha}\big)}\leq C \|\nabla\Delta^{m-1}w\|_{\big(\frac{n}{n-1},\frac{1}{\alpha}\big)}\leq C \|f\|_{L(\log L)^\alpha}.$$
For $1<\ell\leq 2m-1$ \eqref{LlogL1} follows from the Sobolev embeddings, see \cite{ON}.
\end{proof}

\subsection*{Other useful results}

A proof of the results below can be found in \cite{mar1}.
The following Lemma can be considered a generalized mean value identity for polyharmonic function.

\begin{lemma}[Pizzetti \cite{Piz}]\label{lemmapiz} Let $u\in C^{2m}(B_R(x_0))$, $B_R(x_0)\subset \R{n}$, for some $m,n$ positive integers. Then there are positive constants $c_i=c_i(n)$ such that
\begin{equation}\label{pizzetti}
\Intm_{B_R(x_0)}u(x)dx =\sum_{i=0}^{m-1}c_iR^{2i}\Delta^i u(x_0)+ c_m R^{2m}\Delta^m u(\xi),
\end{equation}
for some $\xi \in B_R(x_0)$.
\end{lemma}

\begin{prop}\label{c2m} Let $\Delta^{m}h=0$ in $B_{2}\subset\R{n}$. For every $0\leq \alpha<1$, $p\in [1,\infty)$ and $\ell\geq 0$ there are  constants $C(\ell,p)$ and $C(\ell,\alpha)$ independent of $h$ such that
\begin{eqnarray*}
\|h\|_{W^{\ell,p}(B_1)}&\leq &C(\ell,p) \|h\|_{L^1(B_2)}\\
\|h\|_{C^{\ell,\alpha}(B_1)}&\leq& C(\ell,\alpha)  \|h\|_{L^1(B_2)}.
\end{eqnarray*}
\end{prop}

A simple consequence of Lemma \ref{lemmapiz} and Proposition \ref{c2m} is the following Liouville-type Theorem.

\begin{trm}\label{trmliou} Consider $h:\R{n}\to \R{}$ with $\Delta^m h=0$ and $h(x)\leq C(1+|x|^\ell)$ for some $\ell\geq 0$. Then $h$ is a polynomial of degree at most $\max\{\ell,2m-2\}$. 
\end{trm}

\end{document}